\documentclass[a4paper, 11pt]{amsart}
\usepackage[text={15cm,24cm}]{geometry}

\usepackage{hyperref}
\usepackage{caption}
\usepackage{subcaption}

\usepackage[linesnumbered,commentsnumbered,ruled,vlined]{algorithm2e}

\SetCommentSty{mycommfont}

\usepackage[colorinlistoftodos,bordercolor=orange,backgroundcolor=orange!20,linecolor=orange,textsize=scriptsize]{todonotes}

\usepackage{amsmath}
\usepackage{mathtools}
\usepackage{amssymb}
\usepackage{amsthm}

\usepackage{enumerate}

\theoremstyle{plain}
\newtheorem{theorem}{Theorem}
\newtheorem{proposition}[theorem]{Proposition}
\newtheorem{lemma}[theorem]{Lemma} 
\newtheorem{corollary}[theorem]{Corollary} 
\newtheorem*{theorem*}{Theorem}
\newtheorem*{corollary*}{Corollary}
\newtheorem*{maintheorem*}{Main Theorem}

\newtheorem*{example*}{Example}

\theoremstyle{definition} 
\newtheorem{definition}[theorem]{Definition}

\newtheorem{remark}[theorem]{Remark}

\makeatletter
\renewcommand{\todo}[2][]{\tikzexternaldisable\@todo[#1]{#2}\tikzexternalenable}
\makeatother

\newcommand{\R}{{\mathbb{R}}}

\newcommand{\Z}{{\mathbb{Z}}}

\newcommand{\A}{\mathbf{A}}
\newcommand{\conv}{\text{conv}}
\newcommand{\M}{\mathcal{M}}

\newcommand{\subdiv}{\Sigma}

\newcommand{\seco}{\text{sc}}
\newcommand{\fil}[1]{T_{\text{min}}({#1})}

\title{The MST-fan of a regular subdivision} 

\author[H.\,Eble]{Holger Eble}
\address{Technische Universit\"at Berlin, Chair of Discrete Mathematics/Geometry}
\email{eble@math.tu-berlin.de}

\begin{document}
\thanks{This work was supported by the
SFB-TRR 195 'Symbolic Tools in Mathematics and their Application' of the German Research
Foundation (DFG)}

\begin{abstract}
The dual graph $\Gamma(h)$ of a regular triangulation $\subdiv(h)$ carries a natural metric structure. The minimum spanning trees of $\Gamma(h)$ recently proved to be conclusive for detecting significant data signal in the context of population genetics. In this paper we prove that the parameter space of such minimum spanning trees is organized as a polyhedral fan, called the MST-fan of $\subdiv(h)$, which subdivides the secondary cone of $\subdiv(h)$ into parameter cones. We partially describe its local face structure and examine the connection to tropical geometry in virtue of matroids and Bergman fans.
\hfill
\end{abstract}

\maketitle

\section{Introduction}\label{sec:intro}

Triangulations, or more generally polyhedral complexes, play a crucial role both in theoretical and applied scientific fields related to mathematics, such as  genetics \cite{BPS2007} and machine learning \cite{machine+learning}. See \cite[Chapter 1]{SantosTriangulations} for an overview of triangulations appearing within mathematics. Given a finite point configuration $\A\subset\Z^n$, in practice one often considers a specific class of subdivisions of $\A$, known as \emph{regular subdivisions}. Their maximal cells can be described as shadows of  $n$-dimensional polytopes in $\R^{n+1}$, namely facets of the upper convex envelope associated to height functions $h\colon \A\to \R$ lifting $\A$ into $\R^{n+1}$. The cellular incidences of a fixed regular triangulation $\subdiv(h)$ of $\A$ are recorded in its dual polyhedral complex, the tropical hypersurface $V(H)$ given by the tropical polynomial
\[H(x):=\bigoplus_{a\in \A}h(a)\odot x^a:= \max_{a\in \A}\{h(a)+\langle a,x\rangle\}\enspace,\] 
which is defined to be the set of points  $x\in\R^n$ where at least two of the linear forms $h(a)+\langle a,\cdot\,\rangle$ of $H$ attain  the maximum $H(x)$, cf. \cite{ETC,MaclaganSturmfels}. The bounded cells of $V(H)$ form the \emph{tight span} of $\subdiv(h)$ \cite{HerrmannJoswig+TightSpans}, whose $1$-skeleton is called the \emph{dual graph} $\Gamma(h)$ of $\subdiv(h)$. The nodes of $\Gamma(h)$, i.e. the $0$-cells of $V(H)$, correspond to maximal cells of $\subdiv(h)$ and their adjacency relations are recorded by the edges $E(\Gamma(h))$, which in addition are naturally weighted: If $s=\conv(v_1,\ldots,v_{n+1})$ and $t=\conv(v_2,\ldots,v_{n+2})$ are adjacent maximal simplices of $\subdiv(h)$, the \emph{lattice length} \cite{JoswigBrodsky} or \emph{tropical edge length} \cite{nabijou2021gromovwitten} of the edge ${(s,t)\in E(\Gamma(h))}$ is given by
\begin{align}\label{edgelength}
 L_h(s,t):= |\tilde{L}_h(s,t)|:=\left|\det\begin{pmatrix}
    1 & v_{1,1} & v_{1,2} & \ldots & v_{1,n} & h(v_1) \\
    1 & v_{2,1} & v_{2,2} &  \ldots & v_{2,n} & h(v_2) \\
    \vdots  & \vdots &  \vdots & \vdots & \vdots &\vdots \\
    1 & v_{n+2,1} & v_{n+2,2} & \ldots & v_{n+2,n}  & h(v_{n+2})
  \end{pmatrix}\enspace \right|\enspace .
  \end{align} 
Now, for a fixed height function  $h\colon \A\to \R$ the greedy algorithm choses some minimum spanning tree $\fil{h}\subset \Gamma(h)$, cf. \cite[Chapter 50]{Schrijver}, and in the sequel of this article we will study the fibers of the operation $\fil{\,\cdot\,}$. 
It is well-known \cite[Section 5.2]{SantosTriangulations} that the parameter space $\R^\A$ for regular subdivisions $\subdiv(h)$ of $\A$ is organized in a nicely behaved union of secondary cones $\seco(h)\subset \R^\A$, called the secondary fan of $\A$. The following Main Theorem is Corollary \ref{thm:Khisafan} and it identifies the parameter space $\mathcal{K}(h)=\cup\mathcal{K}(T)$ of \emph{realizable} \textbf{m}inimum \textbf{s}panning \textbf{t}rees $T\subset\Gamma(h)$, i.e. those trees $T$ with $T=\fil{g}\subset\Gamma(h)$ for some $g\in \seco(h)$, as a conic subdivision $\mathcal{K}(h)$ of $\seco(h)$. We refer to $\mathcal{K}(h)$ as the \emph{MST-fan} structure on $\seco(h)$. 
\begin{maintheorem*} 

Let $\subdiv(h)$ be a regular triangulation of the point configuration $\A$.
The MST-fan $\mathcal K(h)$ of $\subdiv(h)$ is a pure fan in $\seco(h)\subset\R^\A$ with full support. Its maximal cones are in bijection with realizable minimum spanning trees of~$\Gamma(h)$.
\end{maintheorem*}

Let again $h\colon\A\to\R$ be a fixed height function and let $G=G(V,E)=\Gamma(h)$. A collection $\mathcal{C}=\{\mathcal{K}(T^{(\mathcal{C}_i)})\}$ of maximal cones in $\mathcal{K}(h)$ is called \emph{$G$-saturating} if the spanning trees $T^{(\mathcal{C}_i)}$ cover $G$, i.e. if $E=\cup_i E(T^{(\mathcal{C}_i)})$ holds.
Let $C\subset \R^\A$ be the polyhedral complex with cells $\bigcap\mathcal{K}(T^{(\mathcal{C}_i)})$ where  $\mathcal{C}$ runs through the $G$-saturating collections.

\begin{corollary*}  The complex  $C$ naturally embeds into the tropical linear space $\text{trop}(\M)$, where $\M:=\M(G)$ is the cycle matroid of $G$, and equals the Bergman fan of $M$ restricted to valuations given by tropical edge lengths as in equation (\ref{edgelength}). 
\end{corollary*}

\subsection*{Motivation from population genetics} 
Regular subdivisions admit a reformulation in terms of linear optimization: A subdivision $\subdiv$ of a point configuration $\A$ with \emph{support} $|\subdiv|:=\bigcup_{\sigma\in\subdiv}\sigma$ is regular if and only if there is a convex support function $f_\subdiv\colon |\subdiv|\to \R$, whose domains of linearity coincide with the maximal cells of $\subdiv$, cf. \cite[Section 1.F]{BrunsGubeladze}.

This interplay of optimization and polyhedral combinatorics was the grounding for \cite{Eble2020} and \cite{Eble2019}, where the theory of regular subdivisions was applied to statistical population genetics in order to detect and study non-additive mutational gene effects, called \emph{epistasis}. Note that the tropical edge length (\ref{edgelength}) measures the degree to which the points $(x,h(x))$, where $x$ runs through the vertex set of $s\cup t$, fail to lie on a hyperplane in $\R^{n+1}$ and thus provides an adequate measure for epistasis.

Put in a biological wording, the point set $\A$ represents an $n$-loci  biallelic (or multi-allelic)  system of genotypes, i.e. $\A=\{0,1\}^n$  or more generally, $\A$ is the vertex set of a product of simplices, and the function $h\colon \A\to\R$ is a genotype-phenotype map which showcases some experimentally obtained physical quantity. Since in this case the support function $f_\subdiv$ can be chosen to record \emph{fittest populations} on $\A$ with respect to $h$, cf. Section \ref{sec:EpistaticInteractions} and \cite{BPS2007}, so do the maximal cells of $\subdiv(h)$. With a slight modification of the tropical edge length (\ref{edgelength}) named \emph{epistatic weight}, the minimum spanning tree $\fil{h}$ was introduced as \emph{epistatic filtration} in \cite{Eble2019}. 
The edges of $\fil{h}$, especially the \emph{significant} ones satisfying a statistical $p$-test, depict a selective choice of \emph{epistatic interactions} and proved to be able to reveal biological relevant information. 
\subsection*{Outline of the paper} In Section \ref{sec:Regular Subdivisions} we review facts on regular subdivisions and tropical geometry. In Section \ref{sec:MSTFan} we introduce the MST-fan of some given regular subdivision. Section \ref{sec:algo} provides an algorithm for computing the single MST-cones from trees and in Section \ref{sec:groupsandtrees} we discuss the effect of changing the edge order of the spanning trees. In Section \ref{sec:TropGeom} we reinterpret the MST-fan in terms of matroid theory and tropical geometry. Section \ref{sec:EpistaticInteractions} shortly explains some use-case for applying tropical geometry to biology from where the motivation to study MST-fan structures initially rebounded. In Section \ref{sec:computations} we provide some computational results.

I am indebted to Michael Joswig and Marta Panizzut for many helpful discussions.

\section{Regular subdivisions and tropical hypersurfaces}\label{sec:Regular Subdivisions}

Let $\A\subset \R^n$ be a point configuration with point labels $J$. Following \cite[Section 2.3]{SantosTriangulations}, a \emph{polyhedral subdivision} of $\A$ is a collection $C$ of subsets of $J$, such that the convex hulls of $C$  cover $\A$ and intersect nicely. A \emph{regular subdivision} of $\A$ is a polyhedral subdivision of $\A$, which is induced by a height function $h\colon \A\to \R$  taking the upper facets of the lifted configuration $\conv(\A^h):=\conv\{(v,h(v))\colon v\in \A\}$ with the induced label set. For instance, the 3-cube has 74 triangulations  (6 triangulations up to symmetry) and all of them are regular. The 4-cube has 92,487,256  triangulations (247,451 triangulations up to symmetry) and 87,959,448 of them are regular, cf. \cite{Huggins, Pournin} for the original results and  \cite{JordanJoswigKastner} for a recomputation  using parallelized reverse search.

\begin{remark}\label{rmk:ConvexSupportFunction}
The regularity of a polyhedral subdivision can be characterized in terms of \emph{convex functions} as follows, cf. \cite[Section 1.F]{BrunsGubeladze}. For a given convex set $S\subset\R^n$, a function $f\colon S\to \R$ is called \emph{convex} if for any two points $x,y\in S$ the graph of the restriction of $f$ to the line segment $[x,y]\subset S$ lies below the line segment $[(x,f(x)),(y,f(y))]$ in $\R^{n+1}$. Now, a polyhedral subdivision $\subdiv$ of $\A$ is regular precisely if it has a convex support function $f_\subdiv$, i.e. the convex hulls of the maximal cells in $\subdiv$ are the regions of linearity of $f_\subdiv$.
\end{remark}

\subsection{The parameter space for regular subdivision}\label{sec:RegSub}
Let $\subdiv(h)$ be a regular triangulation of $\A$. The parameter space of all $g\in \R^\A$ inducing the same triangulation $\subdiv(g)=\subdiv(h)$ is a full-dimensional cone in $\R^\A$, called \emph{secondary cone} $\seco(h)$ of $\subdiv(h)$, cf. \cite[Section 5.2]{SantosTriangulations}. 
It can be understood by looking at local folding constraints on $h$, which push the lifts $(s,t)^h$ of bipyramids $(s,t)$ in $\A$ to the upper hull of $\A^h$. For a given simplex $s=(v_1,\cdots, v_{n+1})$ in $\A$, the conditions on a height function $h\in \R^\A$ to show $s$ in its induced subdivision $\subdiv(h)$ are installed by the linear \emph{folding constraints} $\Psi_{s,j}\geq 0$ for $v_j\in \A\setminus s$, where 
\begin{align}\label{eq:FoldingConstraints}
\Psi_{s,j}(h):=\text{sign}\big(\det\big(\tilde{L}_h(s)\big)\big)\cdot \det \begin{pmatrix}
    1 & v_{1,1} & v_{1,2} & \ldots & v_{1,n} & h(v_1) \\
    1 & v_{2,1} & v_{2,2} &  \ldots & v_{2,n} & h(v_2) \\
    \vdots  & \vdots &  \vdots & \vdots & \vdots &\vdots \\
    1 & v_{n+1,1} & v_{n+1,2} & \ldots & v_{n+1,n}  & h(v_{n+1})\\
    1 & v_{j,1} & v_{j,2} & \ldots & v_{j,n}  & h(v_{j})
  \end{pmatrix}
\end{align}
and where $\tilde{L}_h(s)$ denotes the right matrix of (\ref{eq:FoldingConstraints}) with the last row and column omitted. Hence, the full-dimensional cone $H_s:=\bigcap_{v_j\in \A\setminus s} \{\Psi_{s,j}\geq 0\}\subset\R^\A$ is the parameter space for $s$ to appear in the induced triangulation.  Geometrically, this system requires $(\A-s)^h$ to lie below the hyperplane spanned by $s^h$. 

\subsection{Tropical hypersurfaces}

Tropical Geometry is a discrete variant of Algebraic Geometry and is established over the tropical  semifield $(\R\cup\{-\infty\},\oplus, \odot)$, where $a \oplus b := \max(a,b)$ and $a\odot b:=a+b$. Similar to ordinary Algebraic Geometry, tropical zero sets of tropical polynomials are the central geometric objects of study. Given a set of exponents $S\subset\Z^d$ and coefficients $h(u)\in\R$ for $u\in S$, the associated $d$-variate \emph{tropical polynomial} reads
\[F_h:=\bigoplus_{u\in S} h(u)\odot x_1^{u_1} \odot x_2^{u_2}\odot\ldots\odot x_d^{u_d}=\max\limits_{u\in S}\{h(u)+\langle u,x \rangle\}\]
and can be understood as a function $F_h(x)\colon \R^d\to \R$ in $x$. The \emph{tropical hypersurface} $V(F_h)\subset\R^d$ is defined to be the locus where $F_h$ tropically vanishes, i.e. $V(F_h)$ is the set of all points $x\in\R^d$ where at least two tropical summands of $F_h$ evaluate at $x$ to $F_h(x)$.

\begin{proposition} Let $(s,t)$ be a bounded edge of $V(F_h)$. Then the tropical edge length $\ell_h(s,t)$ equals the euclidean distance between the two $0$-cells  $s$ and $t$ of $V(F_h)$.
\end{proposition}
An important result of tropical geometry states that the $k$-dimensional cells of the hypersurface $V(F_h)$  are in bijection with the $(n-k)$-dimensional cells of the regular subdivision $\subdiv(h)$ for $0\leq k<n$, see \cite[Theorem 1.13]{ETC} for instance. This bijection is inclusion reversing and hence establishes a duality between the combinatorics of $V(F_h)$ and the combinatorics of $\subdiv(h)$. The bounded subcomplex of $V(F_h)$ is called \emph{tight span} of $\subdiv(h)$, cf. \cite{HerrmannJoswig+TightSpans}. The tight span showcases the adjacency relations of interior cells of $\subdiv(h)$ as will be indicated in Figures \ref{fig:trophyp} and \ref{fig:khan3dim}. The $1$-cells, i.e. the edges, of the tight span yield the dual graph $\Gamma(h)$ of $\subdiv(h)$ with natural edge weights given by the tropical edge lengths as in (\ref{edgelength}).
\begin{figure}[h!]
\centering
\includegraphics[scale=1.6]{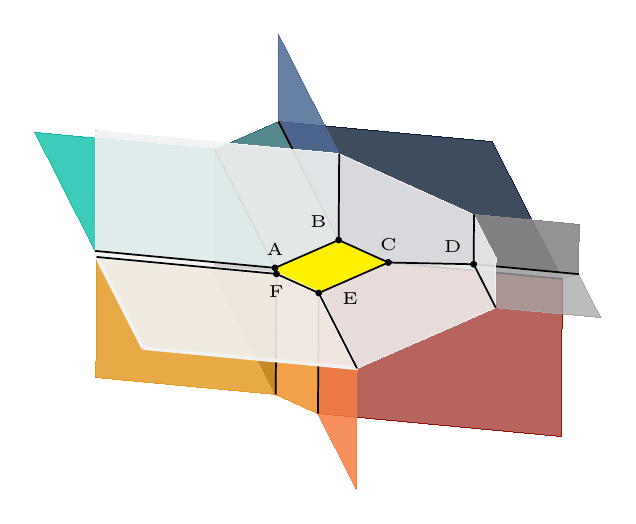}
\caption{The tropical hypersurface associated to a height function on $\A=S=\{0,1\}^3$, which was obtained in a population genetics experiment on Esherichia
coli mutations in the genes topA, spoT, and pykF running Lenski’s long-term evolution experiment, cf.  \cite{Khan1193} and Section \ref{sec:EpistaticInteractions}. The bounded subcomplex of the hypersurface, i.e. the tight span of the regular subdivision $\subdiv(h)$ of $\A$ introduced in \cite{HerrmannJoswig+TightSpans} generalizing the tight spans of finite metric spaces from \cite{Dress} and \cite{Isbell}, has six vertices enumerated $A$ to $F$, six bounded edges $(A,B),(B,C),(C,D),(C,E),(E,F)$ and $(A,F)$ and one bounded $2$-dimensional face colored yellow.}
\label{fig:trophyp}
\end{figure}
\newpage

\begin{figure}[h!]
     \centering
     \begin{subfigure}[b]{0.44\textwidth}
         \centering
         \includegraphics[width=\textwidth]{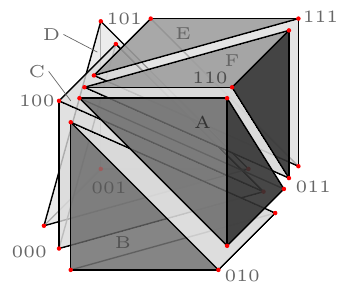}
         \caption{Triangulation $\subdiv(h)$.}
         \label{fig:y equals x}
     \end{subfigure}
     \hfill
     \begin{subfigure}[b]{0.55\textwidth}
         \centering
         \includegraphics[width=\textwidth]{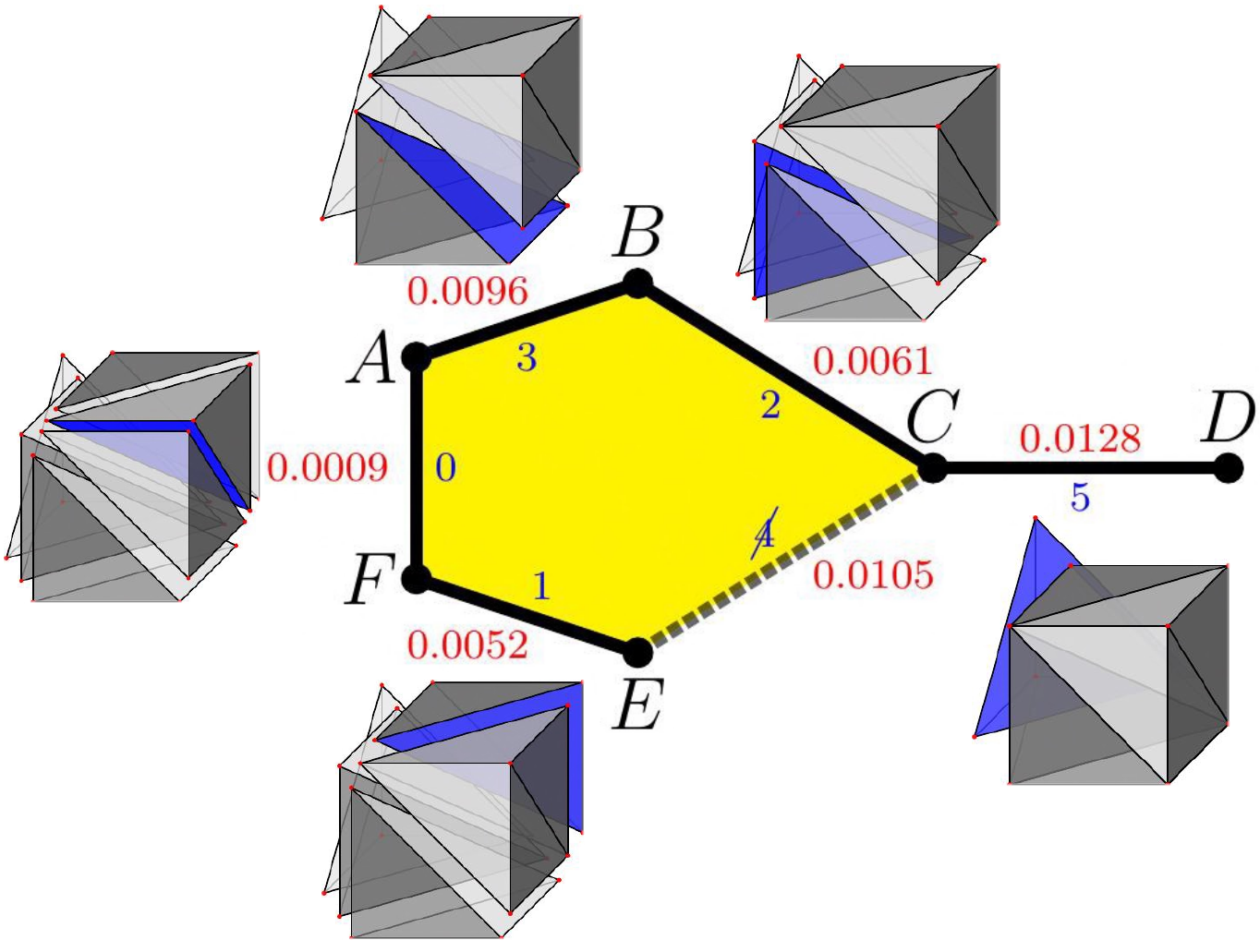}
         \caption{Tight span of $\subdiv(h)$, tropical edge weights (red).}
         \label{fig:three sin x}
     \end{subfigure}
        \caption{Running example.
        (A) The regular triangulation $\subdiv(h)$ dual to Figure \ref{fig:trophyp}. (B) The tight span of $\subdiv(h)$. Its $1$-skeleton, i.e. with the yellow $2$-cell removed, is the dual graph $\Gamma(h)$ of $\subdiv(h)$ and  records the adjacency relations among the maximal simplices of $\subdiv(h)$. 
Tropical edge weights of the dual edges are in red, their order with respect to how the greedy algorithm choses the edges are in blue. 
Note that the edge $(C,E)$ with label $4$ is omitted. Geometrically, the glueing information of this edge is already recorded in the partition $ABCEF|D$ of the cube in (A), cf. Section~\ref{sec:EpistaticInteractions}.
        }
        \label{fig:khan3dim}
\end{figure}

\newpage
\section{The MST-Fan $\mathcal{K}(h)$}\label{sec:MSTFan}
Let again $\A$ be a fixed finite point configuration in $\R^n$. In order to study the class of realizable minimum spanning trees of $\Gamma(h)$ for varying height functions $h\colon\R^\A\to\R$, we first consider the tropical edge length as a linear functional modulo sign.

\begin{definition}\label{def:edgelinearform}
The  \emph{edge length linear form} $\ell_\bullet(r)\colon \R^\A\to\R$ of an $\A$-ridge $r=(s,t)$ is given~by the determinant $\det(\tilde{L}_\bullet(r))$ of the matrix $\tilde{L}_\bullet(r)$ from equation (\ref{edgelength}) in Section \ref{sec:intro}. Thus, evaluating $\ell_\bullet(r)$ at the height function $h\in\R^\A$ yields the tropical edge length $L_h(r)$ up to sign.
\end{definition}

\subsection{Linear forms and circuits}\label{matroids}A matroid $\mathcal{M}=(E,\mathcal{B})$ on a finite ground set $E$ is given by a non-empty collection $\mathcal{B}\subset 2^\M$ of subsets of $E$, called the \emph{set of bases} of $\mathcal{M}$, which fulfills the following \emph{base exchange axiom} (BE), cf. \cite[Section 1.2]{Oxley}:
\begin{align*}
\textbf{(BE) } & \text{For any two members $B_1$ and $B_2$ of $\mathcal{B}$ and any element $x\in B_1\setminus B_2$ there is  } \\ & \text{an element $y\in B_2\setminus B_1$ such that $(B_1\setminus\{x\})\cup \{y\}$ is a member of $\mathcal{B}$\enspace .}
\end{align*} 
A collection $\mathcal{B}$ of bases gives rise to a system $\mathcal{I}\subset 2^\mathcal{M}$ of \emph{independent sets} of $\M$ collecting all subsets of $I\subset E$ lying in some base $B_I\in\mathcal{B}$ and, vice versa, bases are independent sets of maximal cardinality. For our purposes, we consider two matroids. First, given an undirected graph $G=(V,E)$, the \emph{cycle matroid} $\M(G)$ on the ground set $E$ has the spanning trees of $G$ as its bases. Second, the \emph{vector matroid} $\M(\A)$ on the ground set $\A$ depicts affine independences over $\A$, i.e. the bases of $\M(\A)$ are full-dimensional simplices spanned by points of $\A$.


For a given $\A$-ridge $r$, the coefficients $(x_i)_{i\in \A}$ of $\ell_\bullet(r)$ regarded as a row vector in $\R^\A$ are obtained by Laplacian expansion of $\tilde{L}_\bullet(r)$ along its $i$-th row and last column. The support $\text{supp}(\tilde{L}_\bullet(r)):=\{i\in \A\colon x_i\neq 0\}$ is seen by the matroid $\M(\A)$ as we will show~now.

\begin{lemma}\cite[Corollary 1.2.6]{Oxley}
\label{lmm:fundamental_circuit} Every $\A$-ridge $r=(s,t)$ contains a unique circuit $Z^r$.
\end{lemma}
The circuit $Z^r$ of Lemma \ref{lmm:fundamental_circuit} is called the \emph{fundemental circuit} of $r$.
Now, fix an $\A$-ridge $r=(s,t)$ and denote the coefficients of the linear form $\ell_\bullet(r)$ by $(d_v)_{v\in \A}$.

\begin{proposition}\label{lmm:circuitsupport} The circuit $Z^r$ equals the support $
\{v\in\A\colon d_v\neq 0\}$ of $\ell_\bullet(r)$.
\end{proposition}
\begin{proof} If $d_v=0$ holds, then $\text{ker}(\ell_\bullet(r))$ yields relations in $Z^r\subset(s\cup t)\setminus \{v\}$. If $d_v\neq 0$ holds, then $Z^r\nsubseteq \left((s\cup t)\setminus\{w\}\right)\in\mathcal{I}$, and thus we get $v\in Z^r$.
\end{proof}

\begin{remark}\label{samecircuit}

Consequently, one may have $\ell_h(r)=\ell_h(\tilde{r})$ for distinct ridges $r$ and $\tilde{r}$ and generic $h$. In this case the fundamental circuits of $r$ and $\tilde{r}$ need to coincide.
\end{remark}

\begin{proposition} Every circuit $Z$ of $\M(\A)$ equals $Z^r$ for some $\A$-ridge $(s,t)$.
\end{proposition}
\begin{proof}
Let $Z=\{z_1,\ldots, z_r\}\subset \A$. Then $Z\setminus\{z_1\}\subset t$ for some $t\in\mathcal{B}$. Now, choose $z_t\in Z\setminus \{z_1\}$ and set $s:=t\setminus\{z_t\}\cup\{z_1\}$. Assume $s\notin \mathcal \mathcal{I}$. Then there is some circuit $Z^{*}\subset s\cup t$ with $z_t\notin Z^{*}$ contradicting Lemma \ref{lmm:fundamental_circuit}.
\end{proof}

In fact, the fundamental circuit $Z^r$ of the $\A$-ridge $r$ in Proposition \ref{lmm:circuitsupport} is subdivided into a positive and a negative part, $Z^r=(Z^r_+,Z^r_-)$, according to the sign of the coefficients $d_v$. This turns $\M(\A)$ into an \emph{oriented matroid}, cf. \cite{oriented+matroids}. Moreover, the regular triangulations of $\A$ are encoded in the \emph{flip graph}, which is the $1$-skeleton of the \emph{secondary polytope} $\Sigma\text{-poly}(\A)$ of $\A$, cf. \cite[Section 5.3]{SantosTriangulations}. An edge of the flip graph connecting vertices of $\Sigma\text{-poly}(\A)$, i.e. triangulations whose secondary cones $\subdiv(h)$ and $\subdiv(g)$ share a common facet, corresponds to an oriented fundamental circuit $Z^r=(Z^r_+,Z^r_-)$ for some ridge $r=(s,t)\in\Gamma(h)$ and represents a split of the bipyramid $(s,t)$ of $\A$ in two different ways.

\subsection{Reproducing spanning trees of $\Gamma(h)$}
In order to reproduce an edge-ordered spanning tree $T$ of $\Gamma(h)$ as output $\fil{g}=T$ of the greedy algorithm on the weighted dual graph $\Gamma(h)$, one needs to impose conditions on the height function $g$ such that the algorithm is forced to choose the right edges in the correct order.

\begin{lemma}\label{subsec:trivialObservation}
Let $g$ and $h$ be distinct linear forms on $\R^m$. Then \[F^{(g,h)}:=\{x\in\R^m\colon |g(x)|\leq |h(x)|\}\]
is a polyhedral fan consisting of four maximal cones.
\end{lemma}
\begin{proof} 
 The choices for $\sigma_g,\sigma_h\in\{+,-\}$ yield a subdivision of $F^{(g,h)}$ into the four cones

\[F^{(g,h)}_{\sigma_g\sigma_h}\ :=\  \{\sigma_g g\geq 0\}\cap \{\sigma_h h\geq 0\} \cap \{(\sigma_h h-\sigma_g g)\geq 0\}\enspace .\]
\end{proof}

\begin{definition}
An \emph{$\A$-ridge graph} $R$ is a finite set of $\A$-ridges, such that there exists a regular subdivision $\subdiv(h)$ of $\A$ which shows $R$ among its adjacencies,  i.e.  $R\subset\Gamma(h)$ holds.
An \emph{$\A$-ridge forest} is a cycle-free $\A$-ridge graph and an $\A$-ridge tree is a connected $\A$-forest.  \end{definition}

\begin{lemma}\label{lmm:seccone_ridgegraph} For an $\A$-ridge graph $R$, the set of height functions $h\in\R^\A$ such that $R$ occurs in $\Gamma(h)$ is given by the interior of a full-dimensional cone.
\end{lemma}
\begin{proof} Following  \ref{sec:RegSub}, we see $R\subset\Gamma(h)$ precisely if $h$ lies in the interior of $\bigcap_{s\in V(R)} H_s$.
\end{proof}

The parameter cone of Lemma \ref{lmm:seccone_ridgegraph} is the union of those secondary cones whose triangulations realize the adjacencies of $R$  and we define it itself as the \emph{secondary cone} $\seco(R):=\bigcap_{s\in V(R)} H_s$ of $R$. For instance, if $R$ is a spanning tree of $\Gamma(h)$, then $\seco(R)=\seco(h)$ holds since all maximal cells of $\subdiv(h)$ are required to appear. Yet, the graph $R$ can be chosen significantly smaller in order to fulfill $\seco(R)=\seco(h)$, for instance in Figure \ref{fig:khan3dim} the ridge graphs $\{2,3\},\{1,5\},\{1,3\}$ and $\{2,4,5\}$ are inclusion minimal  with this property.

\begin{proposition}\label{th:sign} Given an $\A$-ridge $r\in\Gamma(h)$, the map 
$\text{sign}(\ell_\bullet(r))$ is constant and non-zero on the interior $\text{int}(\seco(R))$ of $\seco(R)$ for any $\A$-ridge graph $R$ containing $r$.
\end{proposition}
\begin{proof} If there are $h,h'\in\R^\A$ with $\text{sign}(\ell_h(r))<0<\text{sign}(\ell_{h'}(r))$, then there is some $g\in \conv\{h,h'\}\subset \seco(R)$ with $\ell_g(r)=0$ contradicting Lemma \ref{lmm:seccone_ridgegraph}.
\end{proof}

For an $\A$-ridge $r\in R$, set $\sigma_r\in \{-1,1\}$ to be $\text{sign}(\ell_\bullet(r))$ on $ \text{int}(\seco(R))$.

\begin{proposition}\label{lmm:uniquecone}
For $\A$-ridges $r_1,r_2\in R$ of an $\A$-ridge graph $R$, we have
\[\#\left \{c\in F^{(l(r_1),l(r_2))}\colon \dim(c\cap \seco(R))=\#\A\right\}\leq1\enspace .\]
If this set has a unique full-dimensional cone, we denote it by $F^{r_1,r_2}(R)$.
\end{proposition}
\begin{proof}
By Proposition \ref{th:sign}, the only candidate for $c$ is $F^{(l(r_1),l(r_2))}_{\sigma_{r_1}\sigma_{r_2}}$. It  appears in full dimension precisely if $\{\sigma_{r_2} l(r_2)-\sigma_{r_1} l(r_1)\geq 0\}\cap \seco(R)$ is full-dimensional. 
\end{proof}

\begin{remark}\label{rem:uniquecone} 
The proof of Proposition \ref{lmm:uniquecone} shows that the hyperplane $\sigma_{r_2} l(r_2)=\sigma_{r_1} l(r_1)$, provided that it  touches the interior of $\seco(R)$, subdivides $\seco(R)$ into the two subcones $\ell_\bullet (r_1)\leq \ell_\bullet (r_2)$ and $\ell_\bullet (r_1)\geq \ell_\bullet (r_2)$.
\end{remark}

With regard to Proposition \ref{lmm:uniquecone}, we define an \emph{ordered $\A$-ridge graph} $R_<$ to be an $\A$-ridge graph $R$ equipped with an enumeration of its edges by an ascending sequence of subsequent integers.  Further, we call an ordered spanning tree $T_<$ of the dual graph $\Gamma(h)$ \emph{realizable} if there exists a height function $g\in\seco(h)$, which satisfies $T_<=\fil{g}$. Note that for a fixed spanning tree, one  order may be realizable while the other is not  as we will see in Figure \ref{fig:Khan3d_coveringbyKruskalcones}. A ridge $r=(s,t)$ in the dual graph  $\Gamma(h)$ is called \emph{stable} if no other ridge of $\Gamma(h)$  
generically, i.e. for a generically chosen $h\in\R^\A$, has the same tropical edge length as $r$. Otherwise, the ridge $r$ is called \emph{instable at some circuit $Z$}, or \emph{$Z$-instable}, and its fundamental circuit $Z$ needs to occur within some other ridge in this case, cf. Lemma \ref{lmm:fundamental_circuit} and Remark \ref{samecircuit}. For instance, if $Z$-instable ridges do not receive subsequent integers in the edge order, the tree is not realizable. Hence any set of $Z$-instable ridges has an interval of integers as \emph{instability range}.

\begin{definition} For a generic height function $h\in\R^\A$ and a realizable ordered spanning tree $T_<$ of $\Gamma(h)$, let the \emph{MST-cone} $\mathcal{K}(T_<)$ of $T_<$ be given by 
\[\mathcal{K}(T_<):=\{g\in \seco(h): g \text{ is generic and } \fil{g}=_\lozenge T_<\}\subset\R^\A\]
as the set of all height functions $g$ on $\A$ such that the 
greedy algorithm on $\Gamma(g)$ possibly  chooses $T_<$ as minimum spanning tree. This is reflected in the notation $\fil{g}=_\lozenge T_<$ as adaption from modal logic. 
\end{definition}

The $\mathcal{K}$ in the notation of the MST-cone goes back to Kruskal, who described in \cite{Kruskal} the greedy algorithm in the graph context, today known as Kruskal's algorithm, with regard to the travelling salesperson problem (TSP). 
In fact, minimum spanning trees have been used ever since in heuristics and approximations for the TSP, see e.g. \cite{KarpI,KarpII}. In the current  paper, we only consider edge weights which are dictated by geometric relations in tropical hypersurfaces. As such, the framework as described in Definition \ref{def:edgelinearform} is linear.

In the sequel we will show that $\mathcal{K}(T_<)$ is a full-dimensional polyhedral subcone of the secondary cone $\seco(h)$ of $\subdiv(h)$. By Proposition \ref{lmm:uniquecone},   every edge of $T_<$ contributes at most one linear hyperplane  to $\mathcal{K}(T_<)$ and the MST-cones give rise to a fan structure.

\section{An algorithm for computing $\mathcal{K}(T_<)$}\label{sec:algo}

Let us fix a realizable ordered spanning tree $T_<:=(r_1,\ldots , r_{m})\subset\Gamma(h)$ for some generic height function $h\in \R^\A$. For $v=0,\ldots, m$ we define  $T_v:=(r_1,\cdots,r_v)$ to be the tree with the same node set as $T$ but only with edges $r_1,\ldots, r_v$ inserted. Further, we set
\[c(T_v):=\{e\in E(\Gamma(h))\setminus E(T_v)\colon T_v\cup e \text{~contains a circuit}\}\]
and we define the \emph{$r_v$-relevant edges} $H(r_v)$ to be 
\[H(r_{v}):=E(\Gamma(h))-E(T_{v})-c(T_v)\enspace .\] 
Thus, the edges in $H(r_v)$ are precisely the edges in $\Gamma(h)$ which do not close a circuit in $T_v$. We say that an edge $r\in\Gamma(h)$ is \emph{relevant} for $r_v$ if $r$ lies in $H(r_v)$. Further, for every  edge $r\in E(\Gamma(h))\setminus T$ there is a unique minimal \emph{relevance index} $i(r)<m$  such that  $r$ is relevant for $T_{i(r)-1}$ but irrelevant for $T_{i(r)}$. In this case we call $\underline{r}=r_{i(r)}$ the \emph{cut edge} of $r$. 

\begin{example*}[Example of Figure \ref{fig:khan3dim} continued.]
In Figure \ref{fig:khan3dim}, edge $4$ is relevant for the edges $0,1$ and $2$. Once edge $3$ is added, edge $4$ closes a circuit in the tree $\{0,1,2,3\}$. Thus $i(4)=4$ and edge $3$ is the cut edge for edge $4$.  
\end{example*}

In order to cut the MST-cone $\mathcal{K}(T_<)$ out of $\seco(h)$ along Proposition \ref{lmm:uniquecone}, we need to consider two kinds of conditions. First, the $(m-1)$ \emph{in-tree conditions} are given by \[\ell_{\bullet}(r_1)\leq \ell_{\bullet}(r_2)\leq \ell_{\bullet}(r_3)\leq\ \cdots\  \leq \ell_{\bullet}(r_{m-1})\leq \ell_{\bullet}(r_{m})\enspace .\] 
Second, for any $r\in E(\Gamma(h))\setminus T$, we make sure that the greedy algorithm  never choses $r$ by requiring the \emph{cut-edge condition} $\ell_\bullet(\underline{r})\leq \ell_\bullet(r)$. 
As a consequence,  the number of facets of $\mathcal{K}(T)$ is always lower than ${\# E(\Gamma)}$ and in practice, the actual number of facets is  almost always a lot lower since for instance by Remark \ref{samecircuit} some of the in-tree conditions might be redundant due to multiple occurrences of instability ranges.\hfill
%

%

%

\begin{algorithm}[H]
        \SetKwIF{If}{ElseIf}{Else}{if}{then}{elif}{else}{}%
        \DontPrintSemicolon
        \SetKwProg{MSTcone}{MST-cone}{}{}
        \LinesNotNumbered
        \KwIn{
        \begin{enumerate}
         \item[(1)] a generic height function $h\in\R^\A$
         \item[(2)] the secondary cone $\seco(h)$ of $h$
         \item[(3)] a realizable ordered spanning tree $T_<=(r_1,\ldots, r_m)\subset\Gamma(h)$
         
        \end{enumerate} 
        }
        \KwOut{the MST-cone $\mathcal{K}(T_<)=\{g\in\R^\A\colon T_{\text{min}}(g)=_\lozenge T_<\}$}
        \MSTcone(){}{
        		\nl $C\longleftarrow \R^{|\A|}$\\
                \nl  \For{$k=1,\ldots,m-1$}{
                        \nl $C\longleftarrow C\cap F^{r_k,r_{k+1}}(T)$\tcp*{Add $(m-1)$ in-tree conditions} 
                        
                }
                \nl  \For{$r\in E(\Gamma(h))\setminus E(T)$}{
                        \nl $C\longleftarrow C\cap F^{\underline{r},r}(T)$\tcp*{Add cut-edge conditions} 
                        
                }

        \nl             \Return{$C\cap \seco(h)$}\label{algo_line:return}
        }
        \caption{Compute the MST-cone $\mathcal{K}(T_<)$ \label{algo:kruskalcone}}
\end{algorithm}

\begin{proposition} Algorithm \ref{algo:kruskalcone} outputs $\mathcal{K}(T_<)$, which is indeed a full-dimensional cone.
\end{proposition}
\begin{proof}
The set of in-tree and cut-edge conditions is necessary and sufficient for $T$ to possibly appear as output of the greedy algorithm,  and defines a cone by Proposition \ref{lmm:uniquecone}.
\end{proof}

%
 
Let $m_*$ be the number $\#E(\Gamma(h))$ of edges of $\Gamma(h)$ and let $v_*$ be the number $\#V(\Gamma(h))$ of its vertices. In the discussion from Section \ref{sec:RegSub} and in the proof of Lemma \ref{subsec:trivialObservation}, we see that the linear forms $l_\bullet(r_1),\ldots,l_\bullet(r_m)$ are byproducts of the secondary cone computation. Concerning the relevance indices needed in the second for-loop of Algorithm \ref{algo:kruskalcone},  every edge which is not in the spanning tree needs at most $m$ cycle checks on subgraphs of $\Gamma(h)$, which takes $\mathcal{O}(m_*+ v_*)$ each using depth-first search. Thus, the complexity of Algorithm \ref{algo:kruskalcone} adds up to $\mathcal{O}(m_*^2+m_*v_*)$. Concerning the quantities $m_*$ and $v_*$, for $\A=\{0,1\}^n$  the $n$-dimensional unit cube the estimate $2^n-n\leq v_*\leq n!$ follows from  \cite[Theorem 2.6.1]{SantosTriangulations}, where $2^n-n = v_*$ holds precisely if $\Gamma(h)$ is a tree, and we have the bound  $m_*\leq \frac{1}{2}(n+1)!-n(2^{n-1}-n +1)$. 
In a variant of Algorithm \ref{algo:kruskalcone} where any ordered input tree $T_<\subset\Gamma(h)$ is allowed, in order to verify its realizability one additionally needs to ensure that the cone $C\cap\seco(h)$ in line \ref{algo_line:return} is full-dimensional.

We remark that if $T_<$ is realizable and has only stable edges as in  Figure \ref{fig:khan3dim}, then the interior $\mathcal{K}(T_<)$ describes the locus where the greedy algorithm has only one  possibility for choosing a minimum spanning tree of $\Gamma(h)$, i.e. $\text{int}(\mathcal{K}(T_<))=\{g\in \seco(h)\colon \fil{g}=T_<\}$. If $T_<$ has instable edges, the greedy algorithm on $g\in\mathcal
K(T_<)$ may possibly choose minimum spanning trees whose edge sets differ from $E(T_<)$. However, the  tree $T_<$ is among the candidates for the optimum and the other outcomes can be controlled.

\begin{proposition} Let $T_< =_\lozenge \fil{h}$ and $T^{'}_<=_\lozenge \fil{h}$ be two valid outputs of the greedy algorithm for the same height function $h$. Then each stable edge of $T_<$ also appears in $T^{'}_<$.
\end{proposition}
\begin{proof}
Assume that the stable edge $r=\{s,t\}\in E(T_<)$ does not appear in $T^{'}_<$, i.e. it closes  a cycle when considered to be added to the forest $\widetilde{T}^{'}_<\subset T^{'}_<$. Then every $s$-$t$ path in $\widetilde{T}^{'}_<$ uses instable edges $r_I$ of $T^{'}_<$ which are not in $T_<$. These edges themselves close cycles in some subtree of $\widetilde{T}_<$ of $T_<$. Replacing each of the $r_I$ by a path in $\widetilde{T}_<$ implies that the nodes $s$ and $t$ belong to the same component of $\widetilde{T}_<$ before $r$ is added. Contradiction.
\end{proof}


 We now define the \emph{MST-fan} $\mathcal{K}(h)$ of $h$ to be the collection of all \emph{MST-cones} $\mathcal{K}(T_<)$ where $T_<$ runs through the realizable ordered spanning trees of $\Gamma(h)$. 

\begin{theorem}\label{thm:Khisafan}
Let $\subdiv(h)$ be a regular triangulation of the point configuration $\A$.
The MST-fan $\mathcal K(h)$ is a pure fan in $\R^\A$ with full support $\text{supp}(\mathcal K(h)):=\bigcup_{c\in \mathcal K(h)} c=\seco(h)$. Two height functions $f,g\in\seco(h)$ are in the same cone of $\mathcal{K}(h)$ if and only if the greedy algorithm possibly produces the same spanning tree on $\Gamma(f)$ and $\Gamma(g)$, i.e. precisely if
 there is an ordered spanning tree $T_<$ of $\Gamma(h)$ with $\fil{f}=_\lozenge T_< =_\lozenge \fil{g}$.
\end{theorem}

\begin{proof}
If $\mathcal{K}(T_<)$ and $\mathcal{K}(T'_<)$ meet, then by Remark \ref{rem:uniquecone} their intersection is induced by some commonly active edge inequalities and therefore it is a face of both $\mathcal{K}(T)$ and $\mathcal{K}(T')$. The bijectivity is a consequence of Proposition \ref{lmm:uniquecone}.
\end{proof}
For further study, we might also consider subfans of the form $\mathcal{K}(T)=\cup \mathcal{K}(T_<)$ collecting all realizable orders on a fixed tree $T$.

\begin{example*}[Example of Figure \ref{fig:khan3dim} continued]
The spanning tree we previously considered reads $T_<=\fil{h}=(0,1,2,3,5)$. Its MST-cone $\mathcal{K}(T_<)$ has four facets given by the four in-tree conditions $0<1<2<3<5$. The facets $(0<1), (3<5)$ and $(2<3)$ meet the interior of $\seco(h)$, whereas  the facet $(1<2)$ lies in the boundary of $\seco(h)$, i.e. no tree with the reversed condition $2<1$ is realizable. In Figure \ref{fig:Khan3d_coveringbyKruskalcones} the dotted lines describe the intersection of $\mathcal{K}(\{0,1,2,3,5\})$ and $\mathcal{K}(\{0,1,3,4,5\})$, which lies in the hyperplane $\ell_\bullet(2)= \ell_\bullet(4)$. The supports of the fans $\mathcal{K}(\{0,1,2,3,5\})$ and $\mathcal{K}(\{0,1,3,4,5\})$ are two adjacent full-dimensional cones. 
\end{example*}


\begin{figure}[h!]
\centering
\includegraphics[scale=0.4]{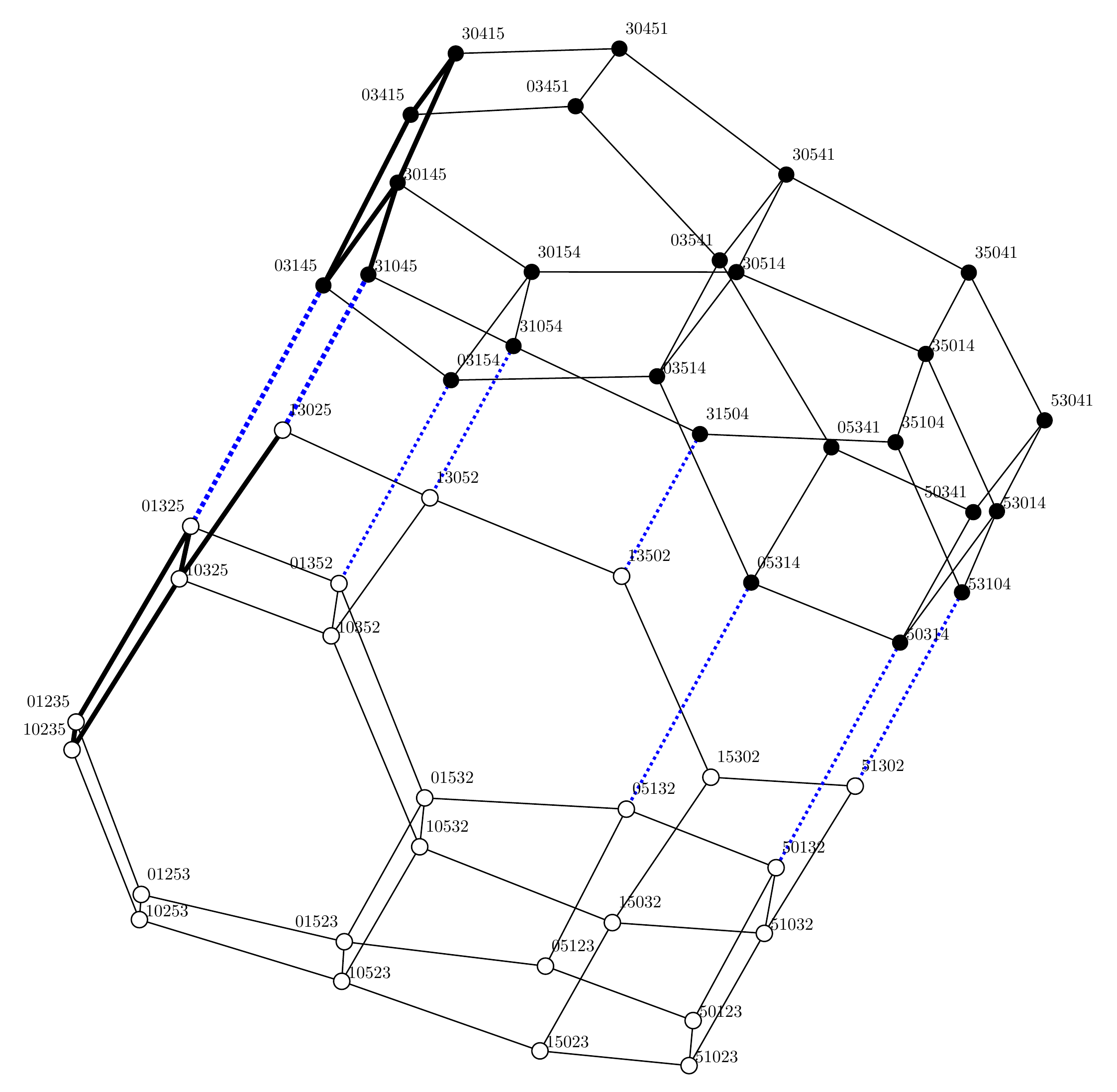}
\caption{The intersection of $\mathcal{K}(\{0,1,2,3,5\})$ and $\mathcal{K}(\{0,1,3,4,5\})$, cf. Figure \ref{fig:khan3dim}. The dual graph $\Gamma(h)$ has five spanning trees, each necessarily featuring edge $5$ and leaving out one of the other five edges. As a computation shows, the only realizable ordered spanning trees arise as orders on the trees $\{0,1,2,3,5\}$ (white nodes) and $\{0,1,3,4,5\}$ (black nodes). Hence, from $5\cdot 5!=600$ possible ordered spanning trees of $\Gamma(h)$, only $\frac{1}{12}\cdot 600=50$ are realizable. The thick subgraph on the upper left corresponds to the subfan of $\mathcal{K}(h)$ leaving edge $5$  fixed and the hyperplane $\ell_\bullet(2)= \ell_\bullet(4)$ is indicated by the dotted blue line segments.
All computations were done using \texttt{polymake}~\cite{polymake}.}
\label{fig:Khan3d_coveringbyKruskalcones}
\end{figure}

%
%

\hfill\\

\section{Permutation groups acting on tree orders}\label{sec:groupsandtrees}

Fix a realizable ordered spanning tree $T_<:=(r_1,\ldots,r_m):=\fil{h}\subset\Gamma(h)$. As explained in Section \ref{sec:MSTFan}, every instable edge $r_k$ of $T_<$ has an interval $I_k:=[i,i+j]$ with $i\leq k\leq i+j$, positive length $j$ and integer endpoints as its  instability range, i.e. where the linear forms $\ell_\bullet(r_i)=\ldots=\ell_\bullet(r_{i+j})$ coincide, cf. Remark \ref{samecircuit}. We may define the instability ranges of stable edges to be the singleton sets depicting their index in the edge order. This convention yields a partition of $\{1,\ldots,m\}$ into an ordered collection $I(T_<)$ of $w\leq m$ instability ranges. Note that in practice, occurences of instability ranges of positive length are frequently encountered in dimension $n\geq 4$, even several of them in the same tree. 

Now, the permutation group $G:=\mathcal{P}(w)$ naturally acts on $I(T_<)$. For a fixed permutation $\sigma\in G$, let $\mathfrak{m}(\sigma)$ be the minimum number of adjacent transpositions that the permutation $\sigma$ can be decomposed into.

\begin{proposition} Let $\sigma\in G$ and let the partitions $I(T_<)$ and $\sigma I( T_<)$ of $\{1,\ldots,m\}$ both be realizable, the latter by the tree $\sigma T_<$. If the MST-cones $\mathcal{K}(T_<)$ and $\mathcal{K}(\sigma T_<)$ intersect in codimension $k$, then $\mathfrak{m}(\sigma)<k+1$ holds. In particular, if the MST-cones $\mathcal{K}(T_<)$ and $\mathcal{K}(\sigma T_<)$ share a common facet, then $\sigma$ is an adjacent transposition. 
\end{proposition} 
\begin{proof}
Each adjacent transposition reverts one non-empty in-tree condition.
\end{proof}

\begin{remark} Let $\tau$ be an adjacent transposition. Then $\dim(\mathcal{K}(T_<)\cap \mathcal{K}(\tau T_<))$ may be of codimension less than one, since there might be cut-edge conditions in the first place which change their direction when reverting an in-tree condition.
\end{remark}

\section{Matroids and Bergman fans}\label{sec:TropGeom}

\subsection{Bergman fans} We follow \cite{ArdilaKlivans} in our presentation. Let $\mathcal{M}=(E,\mathcal{B})$ be a matroid. Any $w\in\R^E$ can be seen as a \emph{valuation} of bases by setting $w(B)=\sum_{b\in B} w(b)$ for any base $B\in\mathcal{B}$. The \emph{initial matroid} $\mathcal{M}_w$ is the matroid on the same ground set $E$ with bases given precisely by those $B\in\mathcal{B}$, whose evaluation $w(B)$ is minimal among all evaluations of bases in $\mathcal{B}$. The \emph{Bergman fan} $\mathfrak{B}(\M)$ is the set of all $w\in \R^E$ such that the initial matroid $\mathcal{M}_w$ has no loops. For instance, if we fix some undirected graph $G$, then $\mathfrak{B}(\M(G))$ is the set of all edge weights $w\in \R^{E(G)}$ such that any edge of $G$ lies in some $w$-minimum spanning tree of $G$. On $\mathfrak{B}(M)$ there are two established fan structures. First, the \emph{coarse subdivision} considers $v,w\in\R^E$ to be equivalent if their initial matroids $\mathcal{M}_v$ and $\mathcal{M}_w$ coincide. It is the coarsest possible fan structure on $\mathfrak{B}(\mathcal{M})$. 
Second, any valuation $w\in \R^E$ gives rise to a \emph{flag} $\mathcal{F}(w):=\{\emptyset=:F_0\subset F_1 \subset \ldots F_{k+1}:=E \}$ of subset of $E$, such that $w$ is constant on each $F_i\setminus F_{i-1}$ and is strictly increasing with $i$, i.e. $w(x)<w(y)$ for any $x\in F_i$ and $y\in F_{i+1}\setminus F_i$. The \emph{fine subdivision} of $\mathfrak{B}(\M)$ considers valuations $v,w\in \R^E$ to be equivalent if their flags $\mathcal{F}(v)$ and $\mathcal{F}(w)$ coincide. The mere set $\mathfrak{B}(\M(G))\subset \R^E$ is also called the \emph{tropical linear space} $\text{trop}(\M)$ of the matroid $\M$.

An important attribute of matroid theory is that there are several ways, called \emph{cryptomorphisms}, to define a matroid, all being equivalent but not in an obvious way. For instance, one may use  \emph{rank functions} to do so. Given a matroid $\M=(E,\mathfrak{I})$ with independence system $\mathfrak{I}$, cf. Section \ref{matroids}, the rank $\rho(A)$ for any subset $A\subset E$ is defined to be the maximal cardinality of an independent set lying inside $A$. A \emph{flat} of $\M$ is a subset of $F\subset E$ such that $\rho(F\cup\{x\})>\rho(F)$ holds for any $x\notin F$. The set of flats of $\M$ forms a geometric lattice which fully describes the combinatorics of $\M$, i.e. giving the lattice of flats is another way to cryptomorphically define matroids.

\begin{proposition}\cite[Theorem 1]{ArdilaKlivans} The fine subdivision realizes the lattice of flags of $\mathcal{M}$.
\end{proposition}
In other words, a flag $\mathcal{F}$ of subsets of $E$ can be written as $\mathcal{F}(w)$ for some $w\in \R^E$ precisely if every member of $\mathcal{F}$ is a flat of $\mathcal{M}$. Hence, the faces of $\mathfrak{B}(\M)$ are given by $\text{pos}(e_{F_0},\ldots, e_{F_{k+1}})$ where $F_0\subset\ldots\subset F_{k+1}$ is a chain of flats of $\mathcal{M}$.

\subsection{The MST-fan as part of a Bergman fan} 
Let again $\A\subset \R^n$ be a finite point configuration and let $h\colon \A\to\R$ be a height function on $\A$. We consider the dual graph $G:=\Gamma(h)\subset\subdiv(h)$ and the cycle matroid $\mathcal{M}(G)$. Consider the map
\begin{align*}
l\colon \R^\A&\longrightarrow \R^{E(G)}\\
h&\longmapsto (\ell_h(r))_{r\in E(G)} \enspace . 
\end{align*}

Since the graph $G$ is naturally weighted by the tropical edge lengths, one may consider the restriction $\hat{\mathfrak{B}}(\M(G)):=\{w\in l(\R^\A)\colon \M(G)_w \text{ has no loops}\}$, where we only allow tropical edge lengths as valuations. 

A collection $\mathcal{C}=\{\mathcal{K}(T^{(\mathcal{C}_i)})\}$ of maximal cones in $\mathcal{K}(h)$ is called \emph{$G$-saturating} if the ordered spanning trees $T^{(\mathcal{C}_i)}$ cover $G$, i.e. if $E=\cup_i E(T^{(\mathcal{C}_i)})$ holds.
Let $C\subset \R^\A$ be the polyhedral complex with cells $\bigcap\mathcal{K}(T^{(\mathcal{C}_i)})$ where  $\mathcal{C}$ runs through the $G$-saturating collections.

\begin{corollary}  The map $l$ embeds the complex  $C$ into $\mathfrak{B}(\M)$ with image $\hat{\mathfrak{B}}(\M(G))$. The induced fine subdivision on $\hat{\mathfrak{B}}(\M(G))\subset \mathfrak{B}(\M(G))$ comes from the cones $\mathcal{K}(T_<)$, whereas the coarse subdivision is perceived by dropping the order and considering $\mathcal{K}(T)=\cup\mathcal{K}(T_<)$.
\end{corollary} 

\begin{example*}In Figure \ref{fig:Khan3d_coveringbyKruskalcones}, the coarse subdivision induced on $\hat{\mathfrak{B}}(\M(G))$ corresponds to the dotted blue part given by the hyperplane $\ell_\bullet(2)= \ell_\bullet(4)$ and is therefore trivial. There are eight distinct dotted blue line segments and they constitute the fine subdivision.
\end{example*}

\section{Application in population genetics: epistasis and spanning trees}\label{sec:EpistaticInteractions}
\subsection{Regular subdivisions and epistatic weights}
Epistatic filtrations are introduced in \cite{Eble2020} and \cite{Eble2019} to track significant epistatic interactions in a fitness landscape $(\A,h)$. There, the function ${h\colon\A\to\R}$ is an experimentally obtained genotype-phenotype map  on a biallelic (or multiallelic) $n$-loci \emph{genotype set} $\A=\{0,1\}^n$, resp.  $\A=\text{vert}(\prod_i \Delta_i)$ in the multiallelic case.  As in the biology literature can be found \cite{Fisher1918}, an affinely independent genotype subsystem $r=\{v_i\colon i\in I\}\subset\A$ is called an \emph{epistatic interaction} if the lifted genotype set $r^h=\{(v_i,h(v_i))\colon i\in I\}\subset \A^h\subset\R^{n+1}$ is affinely independent as well.

A special kind of epistatic interactions can be indicated by $\A$-bipyramids, or $\A$-ridges, i.e. pairs $(s,t)$ of adjacent $n$-simplices $s = \conv \{v_1,\ldots,v_{n+1}\}$ and $t=\conv\{v_2,\ldots,v_{n+2}\}$
with vertices $v_i\in \A$, sharing a common facet. The \emph{epistatic weight} $e_h(s,t)$ of the bipyramid $r=(s,t)$ measures the degree of non-additivity of $h$ on $r$  and is defined in \cite{Eble2019} as
\[e_h(r):=e_h(s,t):= L_h(s,t)\cdot \lambda(s,t)\]
with $\lambda(s,t):=  \text{nvol}(A\cap B)/(\text{nvol}(A)\cdot \text{nvol}(B))$, which doesn't depend on $h$. Thus, the bipyramid $r$ is an epistatic interaction precisely if $e_h(r)>0$. 

Bipyramids occur in a regular triangulation $\subdiv(h)$ as adjacent maximal cells. The transfer from the regular subdivision $\subdiv(h)$ to an according biological statement is now made via the following linear optimization problem $\text{LP}(h,w)$, introduced in \cite{BPS2007}:
\begin{equation}\label{eq:LP}
  \begin{array}{ll}
    \text{maximize} & h \cdot p\\
    \text{subject to} & p\in\Delta_{2^n} \text{ and } \rho(p):=\sum_{v\in\A}p(v)v = w\enspace,
  \end{array}
  \tag{$\text{LP}(h,w)$}
\end{equation}
for some given \emph{allele frequency} $w\in [0,1]^n$ and with decision variable $p$ ranging over all \emph{relative populations} $p\in\Delta_{2^n}=\{p\in \R^\A \colon \sum p_i=1\}$ on  $\A=\{0,1\}^n$, presented here in the biallelic case only. Since $h$ is experimentally obtained and hence generic, there is a unique optimal solution $p^*(h,w)\in\rho^{-1}(w)\subset\Delta_{2^n}$ of (\ref{eq:LP}). The  map 
\begin{align*}
h^*\colon [0,1]^n &\to \R\\
w & \mapsto h\cdot p^*(h,w)
\end{align*}
equals the convex support function $f_{\subdiv(h)}$ from Remark \ref{rmk:ConvexSupportFunction} and therefore its linear regions coincide with the maximal simplices of $\subdiv(h)$. Further, given $s$ and $t$ as above as cells in $\subdiv(h)$,  the quantity $e_h(s,t)$ measures the degree to which the exposed vertex $v_1$ resp. $v_{n+2}$ has a lower than expected phenotype assuming that $h$ extends affinely from $t$ resp. $s$ to the bipyramid $(s,t)$. In this sense $e_h(s,t)$ measures \emph{negative epistasis}, cf. \cite{PhillipsEpistasis}.
The polyhedral subdivision $\{\rho^{-1}(A)\colon A \text{ maximal simplex of } \subdiv(h)\}$ of $\Delta_{2^n}$ now stratifies $\Delta_{2^n}$ into \emph{fittest populations} with respect to varying allele frequencies.

\subsection{Epistatic filtrations}\label{epistaticfiltration}

The \emph{epistatic filtration} of $h$ serves as a tool to separate data noise from significant epistatic signal. It operates on the ascendingly ordered $e_h$-weighted dual graph $\Gamma(h)$. Here, we explicitely interpret the nodes of $\Gamma(h)$ as simplices in $\R^n$.

Geometrically, the set $\sum_0$ of maximal cells of $\subdiv(h)$ yields a decomposition of $[0,1]^n$ into $\#V(\Gamma(h))$ simplicial cells and marks the zero signal starting point of the filtration process.    
Assume $\sum_{k-1}$  is given as a cellular decomposition of the genotope $[0,1]^n$ and the $k$-th edge reads $r_k=(s_k,t_k)$. The decomposition $\sum_{k}$ is now constructed by unifying the cells $c_{ks}$ and $c_{kt}$ of $\sum_{k-1}$ which contain $s_k$ resp. $t_k$ . If one has $c_{ks}\neq c_{kt}$, the edge $r_k$ is called \emph{critical} since the epistatic information can be considered as irredundant in this case. 
The \emph{epistatic filtration} of $h$ is the collection of all critical edges.
In the biological context, late cell agglutinations reflect strong epistatic effects and, since dual graphs are connected, the critical edge with the highest epistatic weight always glues a two-cell split, cf. Figure \ref{fig:khan3dim}(B).

\begin{lemma}
Epistatic filtrations are precisely the minimum spanning trees of $\Gamma(h)$.
\end{lemma}

\begin{proof}
The geometric process depicted in \ref{epistaticfiltration} describes the greedy algorithm on the underlying weighted dual graph $\Gamma(h)$.
\end{proof}

As described in \cite[Section 3.3]{Eble2019}, each epistatic filtration of $h$ gives rise to a unique rooted binary tree $\mathcal{T}(h)$ with leaves labeled by the maximal cells of the regular triangulation $\subdiv(h)$. In the biological context, one refers to such trees with a fixed number of labeled leaves as \emph{phylogenetic trees}. Moreover, there is a natural distance function on the leaf set of $\mathcal{T}(h)$ measuring genetic distances, which is given by the unique path lengths in the corresponding minimum spanning tree inside $\Gamma(h)$. 

\section{Computations}\label{sec:computations}

An extensive study about which trees are realizable as minimum spanning trees inside some dual graph $\Gamma(h)$ demands for an encoding of the regular triangulations of the underlying polytope. Thus, we need to look at classes of polytopes with known secondary fan. These are the \emph{totally splittable polytopes}, introduced and studied in \cite{Totally+splittable+polytopes}, whose regular subdivisions arise exclusively as refinements of splits, i.e. subdivisions with exactly two cells. Totally splittable polytopes comprise simplices, polygons, regular cross polytopes, and prisms and joins over simplices. The code used for the following computations can be found on the author's webpage \url{https://holgereble.github.io/}.

\subsection{Case study I: polygons, }

\begin{proposition} Let $P_n$ be an $n$-gon. The MST-fan structure, neglecting tree orders, is the same as the secondary fan structure.
\end{proposition}
\begin{proof}
As the discussion in \cite[Section 1.1]{SantosTriangulations} shows, the dual graph of a triangulation of $P_n$ is itself a tree.
\end{proof}

\begin{figure}[h!]
 \centering
     \begin{subfigure}[b]{0.3\textwidth}
     \includegraphics[scale=1]{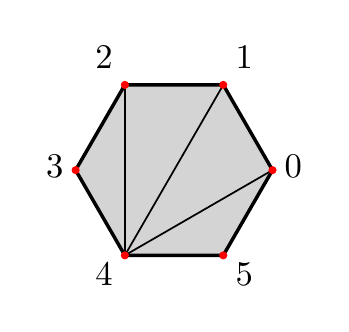}
     \end{subfigure} 
     \begin{subfigure}[b]{0.3\textwidth}
     \includegraphics[scale=1]{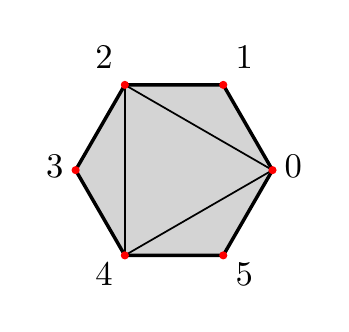}
     \end{subfigure}
     \begin{subfigure}[b]{0.3\textwidth}
     \includegraphics[scale=1]{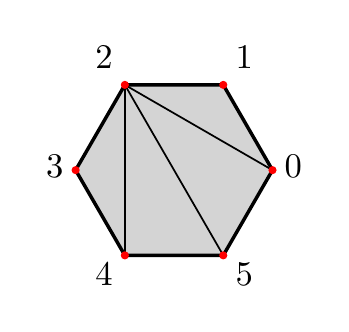}
     \end{subfigure}\caption{Three out of fourteen triangulations of the hexagon. The number of triangulations of the $n$-gon $P_n$ is given by the  Catalan number $C_{n-2}$, cf. \cite{SantosTriangulations}. The underlying dual graphs are trees with $(n-3)$ edges. Each of the $C_4$-many maximal cones of the secondary fan of $P_6$, i.e. the normal fan of a $3$-dimensional associahedron, is split by the MST-fan structure into $3!$ MST-cones according to the permutations of the edge order.} 
\end{figure}

\subsection{Case study II: totally splittable polytopes}

Let $\Pi P_n$ denote the prism $P_n\times I$ over the $n$-gon $P_n$. Moreover, for any polytope $P$ let $\Delta^{\text{reg}}(P)$ be the set of all regular triangulations of $P$ and let $\mathcal{T}(P)$ be the set of all ordered spanning trees that occur inside the dual graph of some regular subdivision of $P$. Similarly, let $\mathcal{T}^{\text{real}}(P)$ be the set of those ordered spanning trees, which occur inside some dual graph of $P$ and which are realizable as minimum spanning trees. The following computations were done in \texttt{polymake} \cite{polymake} on an ordinary office machine (\texttt{Intel Core i7-8700}).

\begin{table}[htpb]
                \[\def\arraystretch{1.5}{
                        \begin{array}{c| r | r | r | r r r r r}                       
                        \text{Polytope} & \#\Delta^{\text{reg}}(P) & \# \mathcal{T}(P) & \#\mathcal{T}^{\text{real}}(P) &  \text{timings} \\
                        \hline
                        \hline
P_5 &  5&  5\cdot 2!=10& 10 & 1s\\
P_6 & 14 & 14\cdot 3!=84 & 84 & 16s\\ 
P_7 & 42 & 42\cdot 4!=1008 & 1008 & 289s\\
P_8 & 132 & 132\cdot 5!=15840 & 15840 & 1.5h\\
\hline
\Pi P_3 &  6& 12 & 12 & 2s\\
\mathbf{\Pi P_4} & 74 &  37632 & 4944 &  3.3h\\
                        \hline
\text{octahedron} &3 &  72 &  24& 15s\\
                        \hline
         
                        \end{array}}
                \]
        \caption{Calculating realizable minimum spanning trees. Polytopes with a more complicated secondary fan structure as the $3$-dimensional cube $\Pi P_4$, which is not totally splittable, feature many non-realizable spanning trees, cf. Figure \ref{fig:Khan3d_coveringbyKruskalcones}.}
        \label{tab:anotherTable}
\end{table}


%
%

%
%

%
%
%
%
%
%
%
%
%
%
%
%
%
%
%
%
%
%
%


\appendix

\end{document}